\documentclass[12pt,reqno,a4paper]{amsart}
\usepackage{amsfonts,amsmath,times,amsthm,amssymb}
\usepackage{tikz}

\newtheorem{theorem}{Theorem}
\newtheorem{lemma}{Lemma}
\newtheorem{prop}{Proposition}
\newtheorem{cor}{Corollary}
\newtheorem{rem}{Remark}

\newcommand{\ftz}{\footnotesize}
\newcommand{\frct}[2]
{\ensuremath{\raisebox{.3ex}{\ftz #1}\!/\!\raisebox{-0.3ex}{\ftz#2}}}

\newcommand{\F}{\mathbb{F}}
\newcommand{\G}{\mathbb{G}}
\newcommand{\Prob}{\mathbb{P}}
\newcommand{\given}{\, | \,}

\newcommand{\Biggiven}{\, \Big{|} \,}

\newcommand{\wigB}{\mathcal{B}}
\newcommand{\wigW}{\mathcal{W}}
\newcommand{\wigR}{\mathcal{R}}
\newcommand{\wigC}{\mathcal{C}}
\newcommand{\Expe}{\mathbb{E}}
\newcommand{\Var}{\mathbb{V}{\rm ar}}

\newcommand{\indicator}{\mathbb{I}}

\newcommand{\covP}{\overset{P}{\longrightarrow}}
\newcommand{\almostsure}{\overset{a.s.}{\longrightarrow}}
\newcommand{\covD}{\overset{D}{\longrightarrow}}
\newcommand{\normal}{\mathcal{N}}
\newcommand{\vecM}{{\bf M}}
\newcommand{\vecX}{{\bf X}}
\newcommand{\vecR}{{\bf R}}

\newcommand{\mSigma}{{\bf \Sigma}}
\newcommand{\maP}{{\bf P}}
\newcommand{\mQ}{{\bf Q}}
\newcommand{\mK}{{\bf K}}
\newcommand{\mA}{{\bf A}}
\newcommand{\mDelta}{{\bf \Delta}}
\newcommand{\Cov}{\mathbb{C}{\rm ov}}

\newcommand{\polya}{P\'{o}lya}

\begin{document}
\begin{center}
	{\Large \bf  
		On nodes of small degrees and degree profile in preferential dynamic attachment circuits}
	
	\bigskip
	{\large \bf Panpan Zhang\footnote{\label{note1}Department of Statistics,
			University of Connecticut, Storrs, CT 06269, U.S.A.} \quad and \quad Hosam M. Mahmoud\footnote{\label{note2}Department of Statistics, The George Washington University, Washington, D.C. 20052, U.S.A.}}
	
	\bigskip
	\bf{\today}
\end{center}

\bigskip\noindent
\begin{center}
{\bf Abstract}
\end{center}
	We investigate the joint distribution of nodes of small degrees and the degree profile in preferential dynamic attachment circuits. In particular, we study the joint asymptotic distribution of the number of the nodes of outdegree $0$ (terminal nodes) and outdegree $1$ in a very large circuit. The expectation and variance of the number of those two types of nodes are both asymptotically linear with respect to the age of the circuit. We show that the numbers of nodes of outdegree $0$ and 1$1$ asymptotically follow a two-dimensional Gaussian law via multivariate martingale methods. We also study the exact distribution of the degree of a node, as the circuit ages, via a series of  \polya-Eggenberger urn models with ``hiccups'' in between. The exact expectation and variance of the degree of nodes are determined by recurrence methods. Phase transitions of these degrees are discussed briefly. This is an extension of the abstract~\cite{ZhangANALCO}.
	

\section{Introduction}
Networks are proliferating all around us. They appear in many forms, such as hardwired,  amorphous cyber and virtual constructs, routes on navigation and trading maps, etc. There is need for models and analysis of networks. In this article, we take up a kind of network that has recently received attention, the {\em preferential attachment circuit}. These are circuits (networks) that grow with newcomers favoring to attach themselves to nodes of higher degrees in the circuit, a manifestation of two principles in social science (e.g.,~\cite{Merton}): ``the rich get richer'' and ``success breeds success.''

In this research, we discuss two properties in preferential attachment circuits: the joint distribution 
of the number of nodes of small degrees and the degree 
profile of a node as the circuit ages. Nodes of small degrees are usually special in random circuits. For example, the nodes of outdegree $0$, known as {\em terminal nodes}, are the outputs in the circuits and have minimal outdegree (i.e., $0$), whereas the nodes of outdegree $1$ are the nodes having only one offspring. 
The study of the distribution of degrees (the {\em degree profile}) in random circuits has been a popular topic in recent research. Knowing the degree of a node can tell, for instance, 
how popular the node is in a social network, or how much demand there is on it in a routing network, which can help allocate the appropriate resources. In the context of circuits, the degree of a node determines the amount of electric current that flows through the node. 

Related circuit models are in~\cite{Tsukiji} (uniform choice of parents), in~\cite{Mahmoud2004} (a uniform positional model), in~\cite{Moler} (a model without the consideration of the dynamical aspect in the present paper). 
The recent literature adds more closely related models. 
The work in~\cite{Berger} considers local limits for preferential attachment graphs, and that in~\cite{Rollin} considers both static and dynamic update schemes, but the focus is on the evolution of the asymptotic degree of a node (they find it,
under appropriate scaling, to be beta distributed), 
as opposed  to our approach to characterize the exact distribution of such a profile; we also aim at comprehending a different kind of profile of node counts, as well. Moreover, the paper~\cite{Resnick} deals with more generalized types of parametrized weights (node affinities).
We became inspired to study 
the dynamic model after reading~\cite{Pekoz, Rollin, Ross}. The source~\cite{Mahmoud2014} discusses several variants of these circuits and the conference paper~\cite{ZhangANALCO} studies the distribution of terminal nodes.

In~\cite{Albert,Bollobas}, the authors consider a model related to the one in the present paper, but allows self-loops. The scope of~\cite{Bollobas} is different from what is covered here; our view is focused on joint distributions among the counts of node degrees. In~\cite{Ostroumova} the authors consider general preferential attachment network models in which $m$ parents
are chosen (according to some general measure of randomness). However, the authors' model is a statistic power of choice. The scope is to study an average profile for that model. The paper is of an empirical nature, where simulation is a tool to substantiate the average results. 
The paper~\cite{Dereich} considers a uniform attachment scheme
for a network grown by adding {\em one} vertex at a time, where every new node is endowed with a ``fitness.'' The models in~\cite{Dereich, Resnick, Wang} are for preferential attachment to a {\em single} parent, where the issue of dynamic versus static choice of parents does not arise.

\section{Preferential Attachment Circuits}
We consider a circuit model (with index $m \ge 1$) that grows in the following way. At the beginning of discrete time (time $0$), there is an originator, which is a single isolated node labeled with $0$. We shall refer to the nodes by their labels. So, the originator is node $0$. At time $n \ge 1$, a new node $n$ appears and attaches itself to $C^{(m)}_{n - 1}$, the circuit existing at time $n - 1$. The circuit $C^{(m)}_n$ grows from its ancestral shape $C^{(m)}_{n - 1}$ in the previous time step $n - 1$ by choosing $m \ge 1$ nodes from $C^{(m)}_{n - 1}$ as {\em parents}. We call the sequence of nodes chosen as parents of node $n$ the $n$th {\em sample}. The sample is taken with replacement. Each of the $m$ parents in the $n$th sample is joined by an edge to node $n$. The choice of the parents is not uniform and the probabilities are dynamic during the construction of a sample. The choice rather depends on the degrees of the nodes present in $C^{(m)}_{n - 1}$. A parent node is chosen with 
probability associated with its degree in $C^{(m)}_{n - 1}$ (more precisely, 
proportional to its outdegree plus one). Say the first of the $m$ 
new edges joins an existing node $j$, 
of outdegree $d$, to node~$n$, 
an event with probability proportional to $d + 1$. 
Once the edge is constructed between nodes $j$ and~$n$, 
it changes the outdegree of node $j$ to $d + 1$, and the chance of node $j$ to be chosen again to construct the second edge as parent of node~$n$ is now proportional to $d + 2$. 
The second parent of node $n$, be it $j$ or some other node, is joined to node~$n$. The process continues in this fashion, adding new edges and accounting for 
the dynamic change in their degrees, reflected in the probabilities 
of choosing existing nodes as parents, till the $m$th member of the sample has been collected and all its members are joined to node $n$.  
At this point, 
we consider that node $n$ has completed its parent selection. Once settled, parents of node $n$ will not be changed. 

Formally, this dynamic insertion scheme is reflected in the following recursive definition. Let the set of vertices of the circuit $C_{n-1}^{(m)}$ be $V_{n-1}^{(m)}$. At step $n$, we add a node to the graph and choose $m$ parents for it in $m$ steps. Let the outdegree of $v\in V_{n-1}^{(m)}$ after the $i$th addition, $i = 0, 1, \ldots, m - 1$, be
$\mbox{outdeg}_{C_{n-1,i}^{(m)}} (v)$. The probability that
$v$ is chosen as the $(i + 1)$th parent for node $n$ is 
$$\frac {\mbox{outdeg}_{C_{n-1,i}^{(m)}} (v) +1} 
{\sum_{x \in V_{n-1}^{(m)}}\left(\mbox{outdeg}_{C_{n-1,i}^{(m)}} (x) +1\right)}.$$

As the sampling is with replacement, a node like $j < n$ 
can be chosen multiple times in the sample at step $n$. For example, $m$ can be equal to five, and node $j$ may appear three times in the sample of parents for node $n$. Hence, three new edges will appear joining the nodes $j$ and $n$. 
The dynamic probabilities are quite sensitive to the order of the sequence of nodes in the sample. For example, the probability of the event that the three choices of node $j$ in the $n$th sample appear as first, second and third may not be the same as its being first, second and fifth. This is a critical element in our probability calculation. As the model allocates higher probability to nodes of higher degrees, the nodes that recruited many children are more likely to attract more children than nodes with fewer children; the model rightly deserves to be described with attributes like ``the rich get richer'' and ``success breeds success.''

Figure~\ref{Fig:circuit} illustrates the growth of a preferential (dynamic) attachment circuit with $m = 3$ during the first two insertions of nodes $1$ and $2$, showing the dynamical changes (edge addition) within each sampling step. From left to right, the probabilities of this particular stochastic path of these networks (conditioning on the previous evolution) are $1$, $1$, $1$, $1$, $\frct 1 5$, $\frct 4 6$, $\frct 2 7$. 

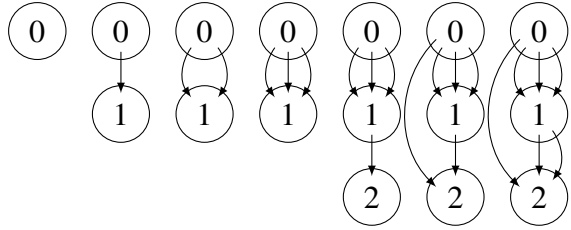
\begin{figure}[tbh]
	\begin{center}
		\tikzstyle{arrow}=[draw, -latex] 
		\begin{tikzpicture}[scale=2.2]
		\draw
		(-1,1) node [circle=0.1,draw] {0};
		\draw
		(-0.5,1) node [circle=0.1,draw] {0}
		(-0.5,0.5) node [circle=0.1,draw] {1};
		\draw[arrow](-0.5,0.875) -- (-0.5,0.625);
		\draw
		(0,1) node [circle=0.1,draw] {0}
		(0,0.5) node [circle=0.1,draw] {1};
		\draw[arrow](0.08, 0.9) to [bend left=35] (0.08, 0.6);
		\draw[arrow](-0.08, 0.9) to [bend right=35] (-0.08, 0.6)  ;
		\draw
		(0.5,1) node [circle=0.1,draw] {0}
		(0.5,0.5) node [circle=0.1,draw] {1};
		\draw[arrow](0.58, 0.9) to [bend left=35] (0.58, 0.6);
		\draw[arrow](0.42, 0.9) to [bend right=35] (0.42, 0.6);  
		\draw[arrow](0.5,0.875) -- (0.5,0.625);
		\draw
		(1,1) node [circle=0.1,draw] {0}
		(1,0.5) node [circle=0.1,draw] {1}
		(1,0) node [circle=0.1,draw] {2};
		\draw[arrow](1.08, 0.9) to [bend left=35] (1.08, 0.6);
		\draw[arrow](0.92, 0.9) to [bend right=35] (0.92, 0.6);  
		\draw[arrow](1,0.875) -- (1,0.625);
		\draw[arrow](1, 0.375) -- (1, 0.125);
		\draw
		(1.5,1) node [circle=0.1,draw] {0}
		(1.5,0.5) node [circle=0.1,draw] {1}
		(1.5,0) node [circle=0.1,draw] {2};
		\draw[arrow](1.58, 0.9) to [bend left=35] (1.58, 0.6);
		\draw[arrow](1.42, 0.9) to [bend right=35] (1.42, 0.6);
		\draw[arrow](1.5,0.875) -- (1.5,0.625);
		\draw[arrow](1.385, 0.95) to [bend right=45] (1.385, 0.06); 
		\draw[arrow](1.5, 0.375) -- (1.5, 0.125);
		\draw
		(2,1) node [circle=0.1,draw] {0}
		(2,0.5) node [circle=0.1,draw] {1}
		(2,0) node [circle=0.1,draw] {2};
		\draw[arrow](2.08, 0.9) to [bend left=35] (2.08, 0.6);
		\draw[arrow](1.92, 0.9) to [bend right=45] (1.92, 0.6);  
		\draw[arrow](2,0.875) -- (2,0.625);
		\draw[arrow](1.885, 0.95) to [bend right=45] (1.885, 0.06); 
		\draw[arrow](2.08, 0.4) to [bend left=35] (2.08, 0.1); 
		\draw[arrow](2, 0.375) -- (2, 0.125);
		\end{tikzpicture}
		\caption{The evolution of a preferential attachment circuit of index $3$ in two steps.}
		\label{Fig:circuit}
	\end{center} 
\end{figure}

In the special case $m = 1$, the circuit is a tree known in the literature as the {\em plane-oriented recursive tree} (PORT), introduced in~\cite{Mahmoud1993}, and its derivatives have received quite a bit of attention in recent years~\cite{Drmota, Fuchs, Mohan, Hwang}. 

It aids the analysis to consider {\em extended circuits}, where each node is supplemented with {\em external nodes} that correspond to the insertion positions of the next entrant. A node with outdegree $s$ is given $s+1$ external positions. 
For a visually pleasing comprehension, we put the external node in 
the ``gaps'' between the edges emanating out of a node. We think of the insertion position to the left (right) of all the edges out of a node as a virtual gap, too. Figure~\ref{Fig:extended} shows the circuits of Figure~\ref{Fig:circuit} after they have been extended. In Figure~\ref{Fig:extended}, the external nodes are shown as squares, some empty and some filled with different colors. This is a color code that will be explained in the sequel.
\begin{figure}[tbh]
	\begin{center}
		\tikzstyle{arrow}=[draw, -latex]
		\begin{tikzpicture}[scale=2.2]
		\draw
		(-1, 1) node [circle=0.1,draw] {0}
		[fill=white] (-1.04, 0.74) rectangle (-0.96, 0.66)
		(-1, 0.875) -- (-1, 0.74);
		\draw
		(-0.5, 1) node [circle=0.1,draw] {0}
		[fill=blue] (-0.74, 0.74) rectangle (-0.66, 0.66)
		(-0.34, 0.74) rectangle (-0.26, 0.66)
		(-0.58, 0.9) -- (-0.7, 0.74)
		(-0.42, 0.9) -- (-0.3, 0.74)
		(-0.5, 0.5) node [circle=0.1,draw] {1};
		\draw[arrow](-0.5, 0.875) -- (-0.5, 0.625);
		\draw
		(0,1) node [circle=0.1,draw] {0}
		(0,0.5) node [circle=0.1,draw] {1};
		\draw[arrow](0.08, 0.9) to [bend left=35] (0.08, 0.6);
		\draw[arrow](-0.08, 0.9) to [bend right=35] (-0.08, 0.6);
		\draw
		[fill=red] (-0.04, 0.74) rectangle (0.04, 0.66)
		(-0.24, 0.94) rectangle (-0.16, 0.86)
		(0.16, 0.94) rectangle (0.24, 0.86);
		\draw
		(-0.125, 1) to [bend right = 15] (-0.2, 0.94)
		(0.125, 1) to [bend left = 15] (0.2, 0.94)
		(0, 0.875) -- (0, 0.74);
		\draw
		(0.5,1) node [circle=0.1,draw] {0}
		(0.5,0.5) node [circle=0.1,draw] {1};
		\draw[arrow](0.58, 0.9) to [bend left=35] (0.58, 0.6);
		\draw[arrow](0.42, 0.9) to [bend right=35] (0.42, 0.6);  
		\draw[arrow](0.5,0.875) -- (0.5,0.625);
		\draw
		[fill=red] (0.4, 0.8) rectangle (0.48, 0.72)
		(0.52, 0.8) rectangle (0.6, 0.72)
		(0.28, 0.94) rectangle (0.36, 0.86)
		(0.64, 0.94) rectangle (0.72, 0.86);
		\draw
		(0.46, 0.88) -- (0.44, 0.8)
		(0.54, 0.88) -- (0.56, 0.8)
		(0.5, 0.375) -- (0.5, 0.24)
		(0.375, 1) to [bend right = 15] (0.32, 0.94)
		(0.625, 1) to [bend left = 15] (0.68, 0.94);
		\draw
		[fill=white] (0.46, 0.24) rectangle (0.54, 0.16);
		\draw
		(1,1) node [circle=0.1,draw] {0}
		(1,0.5) node [circle=0.1,draw] {1}
		(1,0) node [circle=0.1,draw] {2};
		\draw[arrow](1.08, 0.9) to [bend left=35] (1.08, 0.6);
		\draw[arrow](0.92, 0.9) to [bend right=35] (0.92, 0.6);  
		\draw[arrow](1,0.875) -- (1,0.625);
		\draw[arrow](1, 0.375) -- (1, 0.125);
		\draw
		[fill=red] (0.9, 0.8) rectangle (0.98, 0.72)
		(1.02, 0.8) rectangle (1.1, 0.72)
		(0.78, 0.94) rectangle (0.86, 0.86)
		(1.14, 0.94) rectangle (1.22, 0.86)
		(0.96, 0.88) -- (0.94, 0.8)
		(1.04, 0.88) -- (1.06, 0.8)
		(0.875, 1) to [bend right = 15] (0.82, 0.94)
		(1.125, 1) to [bend left = 15] (1.18, 0.94);
		\draw
		[fill=blue] (0.86, 0.24) rectangle (0.94, 0.16)
		(1.06, 0.24) rectangle (1.14, 0.16)
		(0.95, 0.38) -- (0.9, 0.24)
		(1.05, 0.38) -- (1.1, 0.24);
		\draw
		(1.5,1) node [circle=0.1,draw] {0}
		(1.5,0.5) node [circle=0.1,draw] {1}
		(1.5,0) node [circle=0.1,draw] {2};
		\draw[arrow](1.58, 0.9) to [bend left=35] (1.58, 0.6);
		\draw[arrow](1.42, 0.9) to [bend right=35] (1.42, 0.6);  
		\draw[arrow](1.5,0.875) -- (1.5,0.625);
		\draw[arrow](1.385, 0.95) to [bend right=45] (1.385, 0.06); 
		\draw[arrow](1.5, 0.375) -- (1.5, 0.125);
		\draw
		[fill=red] (1.4, 0.8) rectangle (1.48, 0.72)
		(1.52, 0.8) rectangle (1.6, 0.72)
		(1.24, 1.04) rectangle (1.32, 0.96)
		(1.64, 0.94) rectangle (1.72, 0.86)
		(1.26, 0.54) rectangle (1.34, 0.46);
		\draw
		(1.4, 0.92) to [bend right=25] (1.3, 0.54)
		(1.46, 0.88) -- (1.44, 0.8)
		(1.54, 0.88) -- (1.56, 0.8)
		(1.375, 1) -- (1.32, 1)
		(1.625, 1) to [bend left = 15] (1.68, 0.94);
		\draw
		[fill=red] (1.36, 0.24) rectangle (1.44, 0.16)
		(1.56, 0.24) rectangle (1.64, 0.16)
		(1.45, 0.38) -- (1.4, 0.24)
		(1.55, 0.38) -- (1.6, 0.24);
		\draw
		(2,1) node [circle=0.1,draw] {0}
		(2,0.5) node [circle=0.1,draw] {1}
		(2,0) node [circle=0.1,draw] {2};
		\draw[arrow](2.08, 0.9) to [bend left=35] (2.08, 0.6);
		\draw[arrow](1.92, 0.9) to [bend right=35] (1.92, 0.6);  
		\draw[arrow](2,0.875) -- (2,0.625);
		\draw[arrow](1.885, 0.95) to [bend right=45] (1.885, 0.06);
		\draw[arrow](2.08, 0.4) to [bend left=88] (2.08, 0.1) ;
		\draw[arrow](2, 0.375) -- (2, 0.125);
		\draw
		[fill=red] (1.9, 0.8) rectangle (1.98, 0.72)
		(2.02, 0.8) rectangle (2.1, 0.72)
		(1.74, 1.04) rectangle (1.82, 0.96)
		(2.14, 0.94) rectangle (2.22, 0.86)
		(1.76, 0.54) rectangle (1.84, 0.46)
		(2.2, 0.54) rectangle (2.28, 0.46);
		\draw
		(1.9, 0.92) to [bend right=25] (1.8, 0.54)
		(1.96, 0.88) -- (1.94, 0.8)
		(2.04, 0.88) -- (2.06, 0.8)
		(1.875, 1) -- (1.82, 1)
		(2.125, 0.5) -- (2.2, 0.5)
		(2.125, 1) to [bend left = 15] (2.18, 0.94);
		\draw
		[fill=red] (1.86, 0.24) rectangle (1.94, 0.16)
		(2.06, 0.24) rectangle (2.14, 0.16)
		(1.95, 0.38) -- (1.9, 0.24)
		(2.05, 0.38) -- (2.1, 0.24);
		\draw
		[fill=white] (1.96, -0.26) rectangle (2.04, -0.34)
		(2, -0.125) -- (2, -0.26);
		\end{tikzpicture}
		\caption{The evolution of an extended preferential attachment circuit of index $3$ in two steps under a color code.}
		\label{Fig:extended}
	\end{center} 
\end{figure}
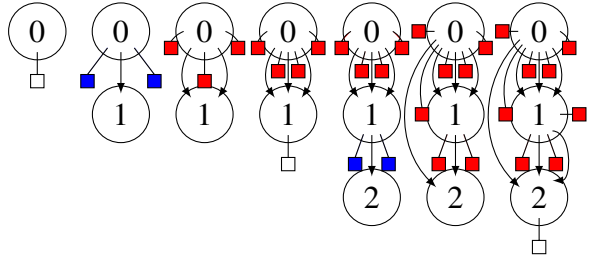

After each node insertion, we add $(m + 1)$ external nodes to the extended circuit, of which $m$ are adjoined to existing nodes, and one is added as a child to the newly inserted node. Therefore, after $n$ node insertions, we have a total of
\begin{equation}
\label{Eq:totalexternal}
\tau_n = (m + 1)n + 1
\end{equation}
external nodes. 

\section{Organization}
The rest of this manuscript is organized as follows. In Section~\ref{Sec:outdegree}, we look into the joint distribution of the numbers of nodes of small degrees. 
The first paragraph of the section introduces a color code of external nodes into white, 
blue and red. The section is divided into two subsections. Subsection~\ref{Subsec:hierarchical} deals with the moments of nodes of outdegree $0$ and $1$. In Subsection~\ref{Subsec:multimartingale}, the multivariate
martingale underlying nodes of outdegree $0$ and $1$ is uncovered and used to establish a bivariate central limit theorem for these nodes. 
We conclude in Section~\ref{Sec:degprofile} by looking 
at the degree of a specified node as the circuit ages. 
We find the exact distribution of that degree 
via a series of \polya-Eggenberger urns with ``hiccups'' in between. The section is also divided into two subsections. In Subsection~\ref{Subsec:Polya}, we say a quick word about \polya-Eggenberger urns and the property of exchangeability for an application in the next section.
We then in Subsection~\ref{Subsec:degree} carry out the computations for the exact probability distribution of the degree. However, that distribution is unwieldy for moment calculations, so we resort to stochastic recurrence to 
find the exact mean and variance of the degree. 
Upon inspecting the average degrees of 
specific nodes, 
we realize that they experience phases through the aging process. 
These phases are discussed briefly at the end of the paper.
\section{Nodes of small outdegrees}
\label{Sec:outdegree}
In this section, we determine the exact first two moments of the number of nodes of
the two smallest outdegrees (i.e., $0$ and $1$). We establish a hierarchical recurrence system to calculate these moments. Later, we conduct an analysis for the asymptotic joint distribution of the number of nodes of outdegrees 0 and 1 via multivariate martingale formulations. Higher moments can be calculated from similar consideration.
However, the calculations become significantly
harder, a manifestation of a phenomenon known as the combinatorial explosion. 
In principle, the method can be generalized to find the joint distribution of the number of nodes of outdegrees $0, 1, 2, \ldots, k$. 
Let $Y_n^{(0)}$ and $Y_n^{(1)}$ be the number of nodes of outdegree 0 and outdegree 1 of a preferential attachment circuit with index $m$ at time $n$, respectively. To study the joint distribution of $Y_n^{(0)}$ and $Y_n^{(1)}$, we rely on external nodes. In addition, we distinguish the external nodes by colors: We color the external nodes emanating from the nodes of outdegree $0$, outdegree $1$, and outdegree greater than $1$ with white, blue, and red, respectively; see Figure~\ref{Fig:extended} for an illustration. Furthermore, we respectively denote the number of white and blue external nodes in the circuit at time $n$ by $W_n$, and $B_n$. The relations between the random variables of interest and the number of white and blue external nodes are
\begin{equation}
Y_n^{(0)} = W_n, \qquad Y_n^{(1)} = \frac 1 2 \, B_n. 
\label{Eq:relationYB}
\end{equation}
Extension to include nodes of higher outdegrees requires more colors.
\subsection{Hierarchical recurrences}
\label{Subsec:hierarchical}
Let $\wigW_{n-1,s}$, $\wigB_{n-1,s}$, and $\wigR_{n-1,s}$ be the number of white, blue, and red external nodes at step $s$ ($0 \le s \le m$) in the $n$th sample, respectively, while the circuit still has $n-1$ internal nodes.  
The network evolves and certain dynamics take place while sampling parents for 
node $n$.
The sampling goes in $m$ steps. 
If a node of outdegee 0 (with one white external node attached to it) is selected at step $s$ of the sampling as parent
of the $n$th node, one edge is constructed to link it to 
node $n$, the outdegree of 
the parent node will be upgraded to 1, with two gaps for blue external nodes (one to the left of the new edge and another one to its right). 
The net result is that the white external node is replaced by two blue external nodes. If, instead, a node of outdegree 1 (with two blue external nodes attached to it)  is selected at step $s$, its outdegree goes up to 2 (defining three gaps); 
these two blue external nodes will be replaced by three red external nodes. The dynamics of red external nodes completely depend on the evolution of white and blue external nodes in the sample. All these dynamics are reflected in the equations:
\begin{align}
	\wigW_{n-1,s} &= \wigW_{n-1,s-1} - \indicator_{n-1,s}^{W},
	\label{Eq:recw}
	\\ \wigB_{n-1,s} &= \wigB_{n-1,s-1} + 2 \indicator_{n-1,s}^{W} - 2 \indicator_{n-1,s}^{B},
	\label{Eq:recb}
\end{align}
where $\indicator_{n-1, s}^{W}$ 
indicates the event that a white external node is selected
at the~$s$th step in the $n$th sample,
while the circuit still has $n-1$ internal nodes. The essence of the definition carries over to  $\indicator_{n-1,s}^{B}$,
the indicator of picking a blue 
external node at the $s$th step in the $n$th sample.
\begin{prop}
	\label{Prop:moments}
	Let $Y_n^{(0)}$ and $Y_n^{(1)}$ be the number nodes of outdegree 0 and~1 in a preferential attachment circuit with index $m \ge 2$ after $n$ insertions. For $n \ge 1$, we have\footnote{The asymptotic notation here is meant to be component-wise, as $n\to\infty$.}
	\begin{align*}
		\begin{pmatrix}
			\Expe[Y_n^{(0)}] 
			\\ \Expe[Y_n^{(1)}] 
		\end{pmatrix} &= 
		\begin{pmatrix}
			\frac{(m + 1)n + m}{2m + 1}
			\\ \frac{(m + 1)(n - 1)\bigl((2m^2 + 2m)n + 2m^2 + m + 1\bigr)}{2(3m + 1)(2m + 1)\bigl((m + 1)n - 1\bigr)} 
		\end{pmatrix} \sim \begin{pmatrix}
			\frac{m + 1}{2m + 1}
			\\ \frac{m(m + 1)}{(3m + 1)(2m + 1)} 
		\end{pmatrix}  n.
	\end{align*}
\end{prop}

\begin{proof}
	We are able to determine the first moment of $\wigW_{n-1,s}$ by solely solving the recurrence equation~(\ref{Eq:recw}). Let
	$\F_{n-1,s-1}$ be the $\sigma$-field generated by the network growth after $n-1$ node insertions 
	and up until the first $s - 1$ steps of sampling parents for the $n$th node, 
	for $1 \le s \le m$.
	We take the conditional expectation (conditioning on $\F_{n-1,s-1}$) to obtain the expression
	\begin{align*}
		\Expe[\wigW_{n-1,s} \given \F_{n-1,s-1}] &= \wigW_{n-1,s-1} - \Expe[\indicator_{n-1,s}^{W} \given \F_{n-1,s-1}] \\
		&=  \wigW_{n-1,s-1} - \frac {\wigW_{n-1,s-1}} {\tau_{n-1}+ s - 1}\\
		&= \frac {(m + 1)(n - 1) + s-1} {(m + 1)(n - 1) + s} \, \wigW_{n-1,s-1}.
	\end{align*}
	Taking another conditional expectation (on $\F_{n-1, s-2}$) , and using the tower property (see Theorem 34.4 on page 448 in~\cite{Billingsley}), we get
	$$
	\Expe[\wigW_{n-1,s} \given \F_{n-1, s-2}] 
	=  \frac {(m + 1)(n - 1) + s-2}{(m + 1)(n - 1) + s} \, \wigW_{n-1, s-2}.$$
	Iterating all the way back to the conditions right before the $n$th sample is collected, where $\wigW_{n-1, 0} = W_{n - 1}$, we obtain
	$$\Expe[\wigW_{n-1,s} \given \F_{n-1,0}] = \frac{(m + 1)(n - 1)W_{n - 1}}{(m + 1)(n - 1) + s}.$$
	Similarly, we are able to construct a recurrence for $\Expe[\wigB_{n-1,s} \given \F_{n - 1,0}]$ with the solution 
	\begin{equation}
	\Expe[\wigB_{n-1,s} \given \F_{n - 1,0}] = \begin{cases}
	\frac{(mn - m + n - 2)B_{n - 1} + 2W_{n - 1}}{mn - m + n }, &   s = 1; \\
	\frac{(m + 1)(n - 1)\bigl((mn - m + n - 2)B_{n - 1} + 2sW_{n - 1}\bigr)}{(mn - m + n + s - 2)(mn - m + n + s - 1)}, &  2 \le s \le  m.
	\end{cases}
	\label{Eq:Port}
	\end{equation}
	For the random variables of interest, $W_n$ and $B_n$, conditioning on $\F_{n - 1, 0}$ yields 
	\begin{align*}
		\Expe[W_{n} \given \F_{n - 1,0}] &= \Expe[\wigW_{n-1,m} \given \F_{n - 1,0}] + 1,
		\\ \Expe[B_{n} \given \F_{n - 1,0}] &= \Expe[\wigB_{n-1,m} \given \F_{n - 1,0}].
	\end{align*}
	In the first recurrence, we have an additional 1 on the right-hand side. The reason is that when a sample is collected, a new node is added to the network and the newcomer automatically
	generates an additional white external node. 
	We dub this phenomenon a ``hiccup,'' which will be discussed 
	in detail in Section~\ref{Sec:degprofile}. Plugging the results for 
	$\Expe[\wigW_{n-1,m} \given \F_{n - 1,0}]$ 
	and $\Expe[\wigB_{n-1,m} \given \F_{n - 1,0}]$ in the recurrences, we arrive at
	\begin{align}
		\Expe[W_{n} \given \F_{n - 1, 0}] &= \frac{(m + 1)(n - 1) W_{n - 1}}{(m + 1)(n - 1) + m} + 1,
		\label{Eq:recW}
		\\ \Expe[B_{n} \given \F_{n - 1, 0}] &= \frac{(m + 1)(n - 1)\bigl((mn - m + n - 2)B_{n - 1} + 2mW_{n - 1}\bigr)}{(mn + n - 2)(mn + n - 1)}.
		\label{Eq:recB}
	\end{align}
	We solve the equations recursively and simultaneously
	under proper initial conditions, 
	and translate the solutions back to $Y_n^{(0)}$ and $Y_n^{(1)}$. We get the stated result for $m \ge 2$.
\end{proof}
\begin{rem}
	\label{Rem:moments}
	In a preferential attachment circuit with index $m = 1$ (i.e., a PORT), the first moments of $W_n$ and $B_n$ are a little different from the expressions in Proposition~\ref{Prop:moments}. Display~(\ref{Eq:Port}) explains the disparity.
	When $m\ge 2$, the case $s=1$ is only a transient state in the sample collection,
	and the circuit comes out of it by the time $\wigW_{n-1,m}$
	is needed; the bottom line in the display is used.
	By contrast, in the case $m=1$, at the end of the sampling, 
	$s = m = 1$, and when the sample is collected, 
	the circuit did not come out of the state $s=1$; the bottom line in the display is used. 
	
	When we translate the first moments of $W_n$ and $B_n$ in a PORT back to $Y_{n}^{(0)}$ and $Y_{n}^{(1)}$, we get
	for $n \ge 1$,
	$$
	\begin{pmatrix}
	\Expe[Y_n^{(0)}]
	\\ \Expe[Y_n^{(1)}] 
	\end{pmatrix} = 
	\begin{pmatrix}
	\frac{2n + 1}{3}
	\\ \frac{n^2 + 2}{3(2n - 1)}
	\end{pmatrix} \sim \begin{pmatrix}
	\frct{2}{3}
	\\ \frct{1}{6} 
	\end{pmatrix} n.
	$$
	This recovers the results of Theorem 2 and Corollary 4 in~\cite{Mahmoud1993}.
\end{rem}
For calculation of higher moments, we appeal to the recurrence relations~(\ref{Eq:recw}) and (\ref{Eq:recb}). We raise them to appropriate powers. 
In this manuscript, we only present the calculations for the second moments for the
random variables $W_n$ and $B_n$, for $m \ge 2$. Higher moments can be obtained via a similar approach. 
Moments of these variables in a random PORT can also 
be calculated by these methods, while
observing, as already noted, that the boundary conditions for the PORT ($m=1$)
are slightly different
from those for $m \ge 1$. 

Towards second moment calculation, we start with squaring the recurrence relations~(\ref{Eq:recw}) and (\ref{Eq:recb}), leading to
\begin{align}
	\wigW^2_{n-1,s} &= \wigW^2_{n-1,s-1} - 2\wigW_{n-1,s-1}\,\indicator_{n-1,s}^{W} + \indicator_{n-1,s}^{W},
	\label{Eq:recwsq}
	\\ \wigB^2_{n-1,s} &= \wigB^2_{n-1,s-1} + 4 \,\indicator_{n-1,s}^{W} + 4\, \indicator_{ n-1+s}^{B}  \nonumber \\
	&\qquad + 4 \, \wigB_{n-1,s-1}\, \indicator_{n-1,s}^{W} - 4\, \wigB_{n-1,s-1}\, \indicator_{n-1,s}^{B}.
	\label{Eq:recbsq}
\end{align}
The term $-8\, \indicator_{n-1,s}^{W}\indicator_{n-1,s}^{B}$ in~(\ref{Eq:recbsq}) vanishes, because~$\indicator_{n-1,s}^{W}$ and $\indicator_{n-1,s}^{B}$ indicate mutually exclusive events. We are able to calculate the second moment 
of~$\wigW_{n-1,s}$ based only on 
the recurrence relation~(\ref{Eq:recwsq}). To calculate the second moment 
of $\wigB_{n-1,s}$, we first need to find the conditional first mixed moment of $\wigW_{n-1,s}\, \wigB_{n-1,s}$, i.e., 
we need to find 
$\Expe[\wigW_{n-1,s}\, \wigB_{n-1,s} \given \F_{n - 1,0}]$. Details are 
presented in the proof of the following proposition.
\begin{prop}
	\label{Prop:samplemoments}
	Let $\wigW_{n-1,m}$ and $\wigB_{n-1,m}$ be the number of white and blue external nodes at step $m$ in the $n$th sample of a preferential attachment circuit with index $m$. We have
	\begin{align*}
		&\Expe[\wigW^2_{n-1,m} \given \F_{n - 1,0}] \\
		&\qquad\qquad\qquad= \frac{(m + 1)(n - 1)\bigl((mn - m + n - 2)W_{n - 1} + m \bigr) W_{n - 1}}{(mn + n - 2)(mn + n - 1)},
		\\ &\Expe[\wigB^2_{n-1,m} \given \F_{n - 1,0}] = \frac{(m + 1)(n - 1)}{3(mn + n - 4)(mn + n - 3)(mn + n - 2)}
		\\ &\qquad\qquad\qquad \qquad\qquad\qquad\times{} \frac{1}{(mn + n - 1)} (\wigC_1 W^2_{n - 1} + \wigC_2 W_{n - 1}B_{n - 1} 
		\\ &\qquad\qquad\qquad\qquad\qquad{} + \wigC_3 B^2_{n - 1} + \wigC_4 W_{n - 1} + \wigC_5 B_{n - 1}),
	\end{align*} 
	where $\wigC_1$, $\wigC_2$, $\wigC_3$, $\wigC_4$, and $\wigC_5$ are 
	constants\footnote{For brevity here,
		we relegate the presentation of these coefficients
		to Appendix~\ref{App:coefficients}.} depending on $m$ and $n$.
\end{prop}

\begin{proof}
	We iterate the recurrence equation~(\ref{Eq:recwsq}) back to the initial condition
	$\wigW^2_{n-1,0}$ $= W^2_{n - 1}$, 
	given $\F_{n - 1,0}$, and obtain
	$$\Expe[\wigW^2_{n-1,s} \given \F_{n - 1,0}] = \frac{(m + 1)(n - 1)\bigl((mn - m + n - 2)W_{n - 1} + s \bigr) W_{n - 1}}{(mn - m + n + s - 2)(mn - m + n + s - 1)},$$
	for all $1 \le s \le m$. 
	Accordingly, we get $\Expe[\wigW^2_{n-1,s} \given \F_{n - 1,0}]$, as stated in the proposition. On the other hand, we take the expectation of both sides of (\ref{Eq:recbsq}) and obtain:
	\begin{align*}
		\Expe[\wigB^2_{n-1,s} \given \F_{n - 1,0}] &= \left(1 - \frac{4}{(m + 1)(n - 1) + s}\right) \Expe[\wigB^2_{n-1,s-1} \given \F_{n - 1,0}] 
		\\ &\qquad{}+ \frac{4\, \Expe[\wigW_{n-1,s-1} \given \F_{n - 1,0}]}{(m + 1)(n - 1) + s} + \frac{4\, \Expe[\wigB_{n-1,s-1} \given \F_{n - 1,0}]}{(m + 1)(n - 1) + s} 
		\\ &\qquad\qquad{} + \frac{4\, \Expe[\wigW_{n-1,s}\wigB_{n-1,s} \given \F_{n - 1,0}]}{(m + 1)(n - 1) + s}.
	\end{align*}
	On the right-hand side of the latter equation, the first term forms the recurrence for $\Expe[\wigB^2_{n-1,s} \given \F_{n - 1,0}]$, and the second and third terms have been obtained in previous calculations, whereas the last term remains unknown. 
	We cross-multiply recurrence equations~(\ref{Eq:recw}) 
	and~(\ref{Eq:recb}) and get
	\begin{align*}
		\wigW_{n-1,s}\, \wigB_{n-1,s} &= \wigW_{n-1,s-1}\, \wigB_{n-1,s-1} - \wigB_{n-1,s-1}\, \indicator^{W}_{n-1,s} + 2W_{n-1,s-1}\, \indicator^{W}_{n-1,s} \nonumber \\
		&\qquad{} - 2\, \indicator^{W}_{n-1,s} - 2\,\wigW_{n-1,s-1}\, \indicator^{B}_{n-1,s}.
	\end{align*}
	Taking expectations, we arrive at
	\begin{align*}
		\Expe[\wigW_{n-1,s}\, \wigB_{n-1,s} \given \F_{n - 1,0}] &= \left(1 - \frac{3}{(m + 1)(n - 1) + s}\right) \\
		&\qquad\qquad\quad\quad {} \times \Expe[\wigW_{n-1,s-1}\, \wigB_{n-1,s-1} \given \F_{n - 1,0}]
		\\ &\qquad\qquad{}+ 2\, \frac{\Expe[\wigW^2_{n-1,s-1} \given \F_{n - 1,0}]}{(m + 1)(n - 1) + s} \\
		&\qquad\qquad {} - 2\, \frac{\Expe[\wigW_{n-1,s-1} \given \F_{n - 1,0}]}{(m + 1)(n - 1) + s}. 
	\end{align*}
	Plugging in the results of $\Expe[\wigW^2_{n-1,s-1} \given \F_{n - 1,0}]$ and $\Expe[\wigW_{n-1,s-1} \given \F_{n - 1,0}]$, we iterate the recurrence equation back to the initial condition
	$\wigW_{n-1,0}\, \wigB_{n-1,0} = W_{n - 1}B_{n - 1}$. 
	We get
	\begin{align*}
		&\Expe[\wigW_{n-1,s}\, \wigB_{n-1,s} \given \F_{n - 1,0}] 
		\\ &\qquad\qquad{}= \frac{(m + 1)(n - 1)(mn - m + n - 2)W_{n - 1} }{(mn - m + n + s -3)(mn - m + n + s - 2)}\\ 
		&\qquad\qquad\qquad\qquad{}\times \frac{1}{(mn - m + n + s - 1)}\bigl( (mn - m + n - 3)B_{n - 1}
		\\ &\qquad\qquad\qquad\qquad\qquad{}+ 2sW_{n - 1} - 2s\bigr).
	\end{align*}
	Then, we are able to solve the recurrence equation for $\Expe[\wigB^2_{n-1,s} \given \F_{n - 1,0}]$. The exact solution is presented in Appendix~\ref{App:coefficients}.  
\end{proof}

\begin{prop}
	\label{Prop:secmoments}
	Let $W_n$ and $B_n$ be the number of white and blue external nodes in a preferential attachment circuit with index $m \ge 2$ after $n$ insertions. For $n \ge 1$, we have
	\begin{align*}
		\Expe[W^2_n] &= \frac{(m + 1)^3 n^3}{(mn + n - 1)(2m + 1)^2} + \frac{1}{(mn + n - 1)(2m + 1)^2(3m + 1)}
		\\ &\qquad{}\times\bigl((8m^2 - m - 1)(m + 1)^2 n^2 + m(m + 1)(3m^2 - 6m - 1)n
		\\ &\qquad\qquad {}- m(2m^3 + 6m^2 + 3m + 1)\bigr),
		\\ \Expe[B_n^2] &= \frac{4m^2(m + 1)^2}{(3m + 1)^2(2m + 1)^2} \, n^2 + \frac{96m^4 + 138m^3 + 83m^2 + 18m + 1}{(5m + 1)(4m + 1)(3m + 1)^2}
		\\ &\qquad{}\times\frac{4(m + 1)m}{(2m + 1)^2} \, n + O(1).
	\end{align*} 
\end{prop}

\begin{proof}
	We calculate the second moment of $W_n$ by squaring both sides of
	the almost-sure relation $W_n = \wigW_{n - 1, m}$, and taking the expectation on both sides; that is
	\begin{align*}
		\Expe[W^2_n \given \F_{n - 1,0}] &= \Expe\bigl[(\wigW_{n-1,m} \given \F_{n - 1,0} + 1)^2 \bigr]
		\\ &= \Expe[\wigW^2_{n-1,m} \given \F_{n - 1,0}] + 2\Expe[\wigW_{n-1,m} \given \F_{n - 1,0}] + 1
		\\ &= \frac{(m + 1)(n - 1)}{(mn + n - 2)(mn + n - 1)} \bigl((mn - m + n - 2)W^2_{n - 1}
		\\ &\qquad {}+ (2mn + m + 2n - 4)W_{n - 1} \bigr) + 1. 
	\end{align*}
	We take another expectation and receive the recurrence for $\Expe[W^2_n]$. 
	We solve it with initial condition $\Expe[W^2_0] = 1$, and we obtain the stated result.
	
	We have a similar argument for $\Expe[B^2_n \given \F_{n - 1,0}]$, namely
	\begin{align*}
		\Expe[B_n^2 \given \F_{n - 1,0}] &= \Expe[\wigB^2_{n-1,m} \given \F_{n - 1,0}] 
		\\ &= \frac{(m + 1)(n - 1)}{3(mn + n - 4)(mn + n - 3)(mn + n - 2)}
		\\ &\qquad \times{} \frac{1}{(mn + n - 1)} (\wigC_1 W^2_{n - 1} + \wigC_2 W_{n - 1}B_{n - 1} 
		\\ &\qquad\qquad{} + \wigC_3 B^2_{n - 1} + \wigC_4 W_{n - 1} + \wigC_5 B_{n - 1}).
	\end{align*}
	We need to calculate $\Expe[W_nB_n]$ before we are able to obtain a clear recurrence equation for $\Expe[B^2_n]$. We multiply the two almost-sure equations 
	$$W_n = \wigW_{n-1,m} \qquad \mbox{and} \qquad B_n = \wigB_{n-1,m},$$	
	and get
	\begin{align*}
		\Expe[W_nB_n \given \F_{n - 1,0}] &= \Expe[(\wigW_{n-1,m} + 1)\wigB_{n-1,m} \given \F_{n - 1,0}]
		\\ &= \frac{(m + 1)(n - 1)(mn - m + n - 2)W_{n - 1} }{(mn + n - 3)(mn + n - 2)(mn + n - 1)}
		\\ &\qquad{}\times \bigl( (mn - m + n - 3)B_{n - 1} + 2mW_{n - 1} - 2m\bigr)
		\\ &\qquad{}+  \frac{(m + 1)(n - 1)}{(mn + n - 2)(mn + n - 1)}
		\\ &\qquad\qquad{}\times\bigl((mn - m + n - 2)B_{n - 1} + 2mW_{n - 1}\bigr).
	\end{align*}
	We take another expectation
	on both sides and solve the recurrence equation for $\Expe[W_n B_n]$ with initial condition $\Expe[W_0 B_0] = 0$, leading to
	\begin{align*}
		\Expe[W_n B_n] &= \frac{(n - 1)}{(4m + 1)(3m + 1)(2m + 1)^2 (mn + n - 2)(mn + n - 1)}
		\\ &\qquad{}\times \bigl(2m(4m + 1)(m + 1)^4 n^3 
		\\ &\qquad\qquad{}+ (8m^3 - 16m^2 + m + 1)(m + 1)^3 n^2
		\\ &\qquad\qquad\qquad{}- (22m^3 + 3m^2 + 13m + 2)(m + 1)^2 n
		\\ &\qquad\qquad\qquad\qquad{}- 2m(m + 1)(6m^3 + 5m^2 + 3m - 2)\bigr).
	\end{align*}
	We plug $\Expe[W_nB_n]$ back into the equation and obtain a clean recurrence equation for $\Expe[B^2_n]$. With the the initial condition $\Expe[B^2_0] = 0$, we get the second moment of $B_n$. An asymptotic
	approximation is stated in the proposition, while the exact solution is provided in Appendix~\ref{App:secondmoment}.
\end{proof}
According to Equation~(\ref{Eq:relationYB}), we obtain the following corollary.
\begin{cor}
	\label{Cor:cov}
	Let $Y_n^{(0)}$ and $Y_n^{(1)}$ be the number of nodes of outdegree 0 and~1 in a preferential attachment circuit with index $m \ge 2$ after $n$ insertions. For $n \ge 1$, 
	we have
	$$\frac{1}{n} \Cov(Y_n^{(0)}, Y_n^{(1)}) \to \begin{pmatrix}
	\frac{2m^2(m + 1)}{(3m + 1)(2m + 1)^2} & -\frac{4(m + 1)^2m^2}{(4m + 1)(3m + 1)(2m + 1)^2}
	\\ -\frac{4(m + 1)^2m^2}{(4m + 1)(3m + 1)(2m + 1)^2} & \frac{2m^2(m + 1)(48m^3 + 59m^2 + 27m + 4)}{(5m + 1)(4m + 1)(3m + 1)^2(2m + 1)^2}
	\end{pmatrix}.$$
\end{cor}

\subsection{Asymptotic joint distribution of the number of nodes of outdegrees~0 and 1}
\label{Subsec:multimartingale}
In this section, we employ the multivariate martingale central limit theorem to study the asymptotic joint distribution of the number of nodes with outdegree $0$ and $1$. Different versions of multivariate martingale central limit theorem have been well developed and used to study \polya\ urn schemes and stochastic processes; see for example~\cite{Kholfi, Kuchler}. For the reader's convenience, we provide a handy reference in the following statement.

\begin{theorem}[Multivariate martingale central limit theorem]  
	\label{Thm:multiMCLT}
	Let $\{\vecM_{n}, \G_{n}\}$ be a zero-mean, square-integrable martingale array with differences $\vecX_{n}$. If for any $\varepsilon > 0$, we have
	\begin{align*}
		\sum_{j = 1}^{n} \Expe \left[\vecX^{\top}_{j} \vecX_{j} \indicator_{\{||\vecX_{j}||_2 > \varepsilon\}} \given \G_{j - 1} \right] &\covP 0,
		\\ \sum_{j = 1}^{n} \Expe \left[\vecX_{j} \vecX^{\top}_{j} \given \G_{j - 1}\right] &\covP \mSigma,
	\end{align*}
	then
	$$\vecM_{n} = \sum_{j = 1}^{n} \vecX_{j} \covD \normal_{k}({\bf 0}, \mSigma),$$
	where $\top$ represents the transpose of a matrix, and $||\cdot||_2$ is 
	the 2-norm of a matrix, and $\normal_{k}({\bf 0}, \mSigma)$ is a $k$--dimensional normal vector with mean ${\bf 0}$ and covariance matrix $\mSigma$.
\end{theorem}

The first in-probability convergence in Theorem~\ref{Thm:multiMCLT} is known as the {\em conditional Lindeberg's condition}; whereas the second in-probability convergence is called the {\em conditional variance condition}.

According to the recurrence relations~(\ref{Eq:recW}) and~(\ref{Eq:recB}), we have
\begin{align*}
	\Expe
	\left[
	\begin{pmatrix}
		W_n
		\\ B_n
	\end{pmatrix}
	\Biggiven
	\F_{n - 1,0} \right]
	&= \begin{pmatrix}
		\frac{(m + 1)(n - 1)}{mn + n - 1} & 0
		\\ \frac{2m(m + 1)(n - 1)}{(mn + n - 2)(mn + n - 1)} & \frac{(m + 1)(n - 1)(mn - m + n - 2)}{(mn + n - 2)(mn + n - 1)}
	\end{pmatrix}
	\\ &\qquad{} \times \begin{pmatrix}
		W_{n - 1}
		\\ B_{n - 1}
	\end{pmatrix}
	+ \begin{pmatrix}
		1
		\\ 0
	\end{pmatrix}.
\end{align*}
Thus, $\left(\begin{smallmatrix}
W_n \\ B_n 
\end{smallmatrix}\right)$ is not a martingale array. To compactify  the notation, let us call the coefficient matrix above $\mA_n$. We manually construct a martingale structure for $\left(\begin{smallmatrix}
W_n \\ B_n
\end{smallmatrix}\right)$ in the following lemma.
\begin{lemma}
	\label{Lem:multimartingale}
	\begin{align*}
		\vecM_n &= \begin{pmatrix}
			\frac{\Gamma \left(\frac{mn + m + n}{m + 1}\right)}{\Gamma(n) \, \Gamma \left(\frac{2m + 1}{m + 1}\right)} & 0
			\\ \\ \frac{2 \, \Gamma\left(\frac{mn + m + n}{m + 1}\right)}{\Gamma(n) \, \Gamma \left(\frac{2m + 1}{m + 1}\right)} - \frac{2(mn + n - 1) \, \Gamma \left(\frac{mn + m + n - 1}{m + 1}\right)}{m \, \Gamma(n) \, \Gamma \left(\frac{2m}{m + 1}\right)} & \frac{(mn + n - 1) \, \Gamma \left(\frac{mn + m + n - 1}{m + 1}\right)}{m \, \Gamma(n) \, \Gamma \left(\frac{2m}{m + 1}\right)}
		\end{pmatrix}
		\\ &\qquad \times \begin{pmatrix}
			W_n
			\\ B_n
		\end{pmatrix} + 
		\begin{pmatrix}
			- \frac{\Gamma \left(\frac{mn + 2m + n + 1}{m + 1}\right)}{\Gamma(n) \, \Gamma \left(\frac{3m + 2}{m + 1}\right)} + 1
			\\ \\ \frac{(m + 1)(2mn + m + 2n - 1) \, \Gamma \left(\frac{mn + 2m + n}{m + 1}\right)}{m(3m + 1) \, \Gamma(n) \, \Gamma \left(\frac{2m}{m + 1}\right)} - \frac{2 \, \Gamma \left(\frac{mn + 2m + n + 1}{m + 1}\right)}{\Gamma(n) \, \Gamma \left(\frac{3m + 2}{m + 1}\right)}
		\end{pmatrix}
	\end{align*}
	is a martingale array for $n \ge 2$.
\end{lemma}

\begin{proof}
	We apply a linear transformation to  $ \vecR_n = \left(\begin{smallmatrix}
	W_n \\ B_n
	\end{smallmatrix}\right)$,
	such that $\vecM_n = \maP_n \vecR_n + \mQ_n$ is a martingale array. By the fundamental property of martingales, we have
	\begin{align*}
		\Expe[\vecM_n \given \F_{n - 1,0}] &= \maP_n \Expe \left[\begin{pmatrix}
			W_n 
			\\ B_n
		\end{pmatrix} \Biggiven \F_{n - 1,0} \right] + \mQ_n\\
		&=\maP_n \left(\mA_n \vecR_{n-1} + \begin{pmatrix} 
			1
			\\ 0
		\end{pmatrix} \right) + \mQ_n\\
		&= \vecM_{n - 1}\\
		&= \maP_{n-1}\vecR_{n-1} + \mQ_{n-1},
	\end{align*}
	for all $n \ge 2$. We equate the coefficients of $\vecR_{n-1}$ 
	on the second and fourth lines of the chain of equalities 
	and get for $\maP_n$ the recursive relation
	$$\maP_n = \maP_{n - 1}\mA_n^{-1} = \maP_{n - 2}\mA_{n - 1}^{-1}\mA_n^{-1} = \cdots = \maP_{1}\mA_2^{-1} \cdots \mA_{n}^{-1},$$
	for an arbitrary choice of $\maP_1$. Without loss of generality, let $\maP_1$ be the identity matrix. We calculate the product of $\mA_n^{-1}$ by applying Jordan normal form to each term, which results in
	$$\mA_n^{-1} = \begin{pmatrix}
	1 & 0 \\ 2 & -2
	\end{pmatrix}
	\begin{pmatrix}
	\frac{mn + n - 1}{mn - m + n - 1} & 0
	\\ 0 & \frac{(mn + n - 1)(mn + n - 2)}{(n - 1)(m + 1)(mn - m + n - 2)}
	\end{pmatrix}
	\begin{pmatrix}
	1 & 0 \\ 2 & -2
	\end{pmatrix}^{-1}.$$
	Thus, we get a solution:
	$$\maP_n = \begin{pmatrix}
	\frac{\Gamma \left(\frac{mn + m + n}{m + 1}\right)}{\Gamma(n) \, \Gamma \left(\frac{2m + 1}{m + 1}\right)} & 0
	\\ \frac{2 \, \Gamma\left(\frac{mn + m + n}{m + 1}\right)}{\Gamma(n) \, \Gamma \left(\frac{2m + 1}{m + 1}\right)} - \frac{2(mn + n - 1) \, \Gamma \left(\frac{mn + m + n - 1}{m + 1}\right)}{m \, \Gamma(n) \, \Gamma \left(\frac{2m}{m + 1}\right)} & \frac{(mn + n - 1) \, \Gamma \left(\frac{mn + m + n - 1}{m + 1}\right)}{m \, \Gamma(n) \, \Gamma \left(\frac{2m}{m + 1}\right)}
	\end{pmatrix}.$$
	On the other hand, we equate the free terms and find
	$$\mQ_n = \mQ_{n - 1} - \maP_n  \begin{pmatrix}
	1 \\ 0
	\end{pmatrix} = \mQ_1 - \sum_{j = 2}^{n} \maP_j  \begin{pmatrix}
	1 \\ 0
	\end{pmatrix},$$
	for arbitrary choice of $\mQ_1$. Let $\mQ_1$ be the $2\times 2$ zero matrix; then we arrive at
	$$\mQ_n = \begin{pmatrix}
	- \frac{\Gamma \left(\frac{mn + 2m + n + 1}{m + 1}\right)}{\Gamma(n) \, \Gamma \left(\frac{3m + 2}{m + 1}\right)} + 1
	\\ \frac{(m + 1)(2mn + m + 2n - 1) \, \Gamma \left(\frac{mn + 2m + n}{m + 1}\right)}{m(3m + 1) \, \Gamma(n) \, \Gamma \left(\frac{2m}{m + 1}\right)} - \frac{2 \, \Gamma \left(\frac{mn + 2m + n + 1}{m + 1}\right)}{\Gamma(n) \, \Gamma \left(\frac{3m + 2}{m + 1}\right)}
	\end{pmatrix}. $$
\end{proof}
Note that  $\vecM_j - \left(\begin{smallmatrix} 1 \\ 0
\end{smallmatrix}\right)$ are centered bivariate martingales.
Let us take
$$\mK_n = \begin{pmatrix}
n^{-\frac{3m + 1}{2(m + 1)}} & 0
\\ 0 & n^{-\frac{5m + 1}{2(m + 1)}}
\end{pmatrix}.$$
The martingale differences are
$$\vecX_j = \vecM_j - \vecM_{j-1}.$$ 
So, $\mK_n \vecM_j$ are centered martingales, 
and  $\mK_n \vecX_j$ are martingale differences, too. 

Next, let us verify the conditional Lindeberg's condition and compute the covariance matrix $\mSigma$ stated in Theorem~\ref{Thm:multiMCLT}.

\begin{lemma} 
	\label{Lem:multiLindeberg}
	We have
	$$ \sum_{j = 1}^{n} \Expe [(\mK_n \vecX_{j})^{\top} (\mK_n \vecX_{j}) \indicator_{\{||\mK_n \vecX_{j}||_2 > \varepsilon\}} \given \F_{n - 1,0} ] \covP 0. $$
\end{lemma}
\begin{proof}
	Use the fact that
	$$|W_j - W_{j - 1}| \le m - 1 \qquad \mbox{and} \qquad |B_j - B_{j - 1}| \le 2m,$$
	to find
	\begin{align*}
		||\mK_n \vecX_j||_2 &\le ||\mK_n\vecM_j - \mK_n\vecM_{j - 1}||_2
		\\ &= ||\mK_n\maP_j \vecR_j + \mK_n\mQ_j - (\mK_n\maP_{j - 1} \vecR_{j - 1} + \mK_n\mQ_{j - 1}) ||_2
		\\ &\le ||\mK_n (\maP_j - \maP_{j - 1})\vecR_j||_2 + ||\mK_n \maP_{j - 1}(\vecR_j - \vecR_{j - 1})||_2 
		\\&\qquad{}+ ||\mK_n (\mQ_j -\mQ_{j - 1})||_2
		\\&= O\bigl(n^{-\frac{m}{m + 1}}\bigr) + O\bigl(n^{-\frac{m}{m + 1}}\bigr) + O\bigl(n^{-1}\bigr)
		\\&\, \longrightarrow \, 0,
	\end{align*}
	as $n \to \infty.$ Thus, the events $\{||\mK_n \vecX_j||_2 > \varepsilon\} $ are empty for all $j$ and $\varepsilon > 0$. The conditional Lindeberg's condition is verified.
\end{proof}
\begin{lemma}
	\label{Lem:multicovariance}
	We have
	$$ \sum_{j = 1}^{n} \Expe [(\mK_n\vecX_{j})(\mK_n\vecX_{j})^{\top} \given \F_{j - 1,0} ] \covP \mSigma, $$
	where 
	$$\mSigma = \begin{pmatrix}
	\frac{2m^2}{(3m + 1) (m + 1) \, \Gamma^2\left(\frac{3m + 2}{m + 1}\right)} & - \frac{12m(m + 1)}{(3m + 1) (4m + 1) \, \Gamma\left(\frac{2m}{m + 1}\right) \, \Gamma\left(\frac{3m + 2}{m + 1}\right)}
	\\  \\ - \frac{12m(m + 1)}{(3m + 1) (4m + 1) \, \Gamma\left(\frac{2m}{m + 1}\right) \, \Gamma\left(\frac{3m + 2}{m + 1}\right)} & 
	\frac{24(m + 1)^3 (7m + 3)}{(3m + 1)^2(5m + 1)(2m + 1) \, \Gamma^2\left(\frac{2m}{m + 1}\right)}
	\end{pmatrix}.$$
\end{lemma}
\begin{proof}
	Noticing that 
	$$(\mK_n \vecX_j)(\mK_n \vecX_j)^\top = \mK_n \vecX_j \vecX^\top_j \mK_n^\top,$$
	we first calculate the asymptotic values for the four entries in the covariance matrix $\mDelta_n = \sum_{j = 1}^{n} \Expe\bigr [\vecX_{j} \vecX^\top_{j}  \given \F_{n - 1,0} \bigr]$ one by one:
	\begin{align*}
		\mDelta_n(1, 1) &= \sum_{j = 1}^{n} \Expe \bigl[ (\maP_j(1, 1) W_j - \maP_{j - 1}(1, 1) W_{j - 1} + \mQ_j(1, 1) 
		\\ &\qquad\qquad{}- \mQ_{j - 1}(1, 1))^2 \given \F_{j - 1, 0}\bigr]
		\\ &= \sum_{j = 1}^{n} \frac{2m^2}{(m + 1)^2 \, \Gamma^2\left(\frac{3m + 2}{m + 1}\right)} \bigl(j^{\frac{2m}{m + 1}} + O\bigl(j^{\frac{2m - 1}{m + 1}}\bigr)\bigr)
		\\ &\sim \frac{2m^2}{(3m + 1) (m + 1) \, \Gamma^2\left(\frac{3m + 2}{m + 1}\right)} \, n^{\frac{3m + 1}{m + 1}}.
	\end{align*}
	The two entires on the antidiagonal of $\mDelta_n$ are the same, i.e., $\mDelta_{12} (n)= \mDelta_{21}(n)$, where
	\begin{align*}
		\mDelta_n(1, 2) &= \sum_{j = 1}^{n} \Expe\bigl[ (\maP_{j}(1, 1) W_j - \maP_{j - 1}(1, 1) W_{j - 1} + \mQ_{j}(1, 1) - \mQ_{j - 1}(1, 1))
		\\ &\qquad\qquad{}\times (\maP_{j}(2, 2) B_j - \maP_{j - 1}(2, 2) B_{j - 1} + \mQ_{j}(2, 1) 
		\\ &\qquad\qquad\qquad{}- \mQ_{j - 1}(2, 1)) \given \F_{j - 1,0}\bigr]
		\\ &= \sum_{j = 1}^{n} -\frac{12m}{(3m + 1) \, \Gamma\left(\frac{2m}{m + 1}\right) \, \Gamma\left(\frac{3m + 2}{m + 1}\right)} \bigl(j^{\frac{3m}{m + 1}} + O\bigl(j^{\frac{2m}{m + 1}}\bigr)\bigr)
		\\ &\sim - \frac{12m(m + 1)}{(3m + 1) (4m + 1) \, \Gamma\left(\frac{2m}{m + 1}\right) \, \Gamma\left(\frac{3m + 2}{m + 1}\right)} \, n^{\frac{4m + 1}{m + 1}}.
	\end{align*}
	We also have
	\begin{align*}
		\mDelta_{n} (2, 2) &= \sum_{j = 1}^{n} \Expe\bigl[(\maP_{j} (2, 1) W_j - \maP_{j - 1}(2, 1)W_{j - 1} + \maP_{j}(2, 2) B_j 
		\\ &\qquad{}- \maP_{j - 1}(2, 2) B_{j - 1} + \mQ_{j}(2, 1) - \mQ_{j - 1}(2, 1))^2 \given \F_{j - 1,0}\bigr]
		\\ &= \sum_{j = 1}^{n} \frac{24(m + 1)^2 (7m + 3)}{(3m + 1)^2(2m + 1) \, \Gamma\bigl(\frac{2m}{m + 1}\bigr)} \bigl(j^{\frac{4m}{m + 1}} + O\bigl(j^{\frac{3m}{m + 1}}\bigr)\bigr)
		\\ &\sim \frac{24(m + 1)^3 (7m + 3)}{(3m + 1)^2(5m + 1)(2m + 1) \, \Gamma^2\left(\frac{2m}{m + 1}\right)} \, n^{\frac{5m + 1}{m + 1}}.
	\end{align*}	 
	Thus, we obtain 
	\begin{align*}
		&\begin{pmatrix}
			n^{-\frac{3m + 1}{2(m + 1)}} & 0
			\\ 0 & n^{-\frac{5m + 1}{2(m + 1)}}
		\end{pmatrix}
		\begin{pmatrix}
			\mDelta_n(1, 1) & \mDelta_n(1, 2)
			\\ \mDelta_n(2, 1) & \mDelta_n(2, 2)
		\end{pmatrix}
		\begin{pmatrix}
			n^{-\frac{3m + 1}{2(m + 1)}} & 0
			\\ 0 & n^{-\frac{5m + 1}{2(m + 1)}}
		\end{pmatrix}^{\top}
		\\ &\qquad \qquad \qquad \almostsure \mSigma,
	\end{align*}
	which is a stronger convergence mode than the one stated in the theorem.
\end{proof}
Recall~(\ref{Eq:relationYB}) relating the target random variables and the external nodes. We determine the asymptotic joint distribution of $Y_n^{(0)}$ and  $Y_n^{(1)}$ in the following theorem.
\begin{theorem} 
	\label{Thm:multimartingaleclt}
	Let $Y_n^{(0)}$ and $Y_n^{(1)}$ respectively be the number of nodes of outdegree 0 and outdegree 1 of a preferential attachment circuit with index $m$ at time $n$. For $m \ge 2$, we have
	\begin{equation*}
		\frac{1}{\sqrt n}\left(\begin{pmatrix}
			Y_n^{(0)}
			\\ Y_n^{(1)} 
		\end{pmatrix}  -  \begin{pmatrix}
			\frac{m + 1}{2m + 1}
			\\ \frac{m(m + 1)}{(3m + 1)(2m + 1)}
		\end{pmatrix} \, n \right)
		\covD\
		\normal_2 \left( {\bf 0}, \tilde{\mSigma} \right),
	\end{equation*}
	where 
	$$ \tilde{\mSigma} = \begin{pmatrix}
	\frac{2m^2(m + 1)}{(3m + 1)(2m + 1)^2} & -\frac{4(m + 1)^2m^2}{(4m + 1)(3m + 1)(2m + 1)^2}
	\\ \\ -\frac{4(m + 1)^2m^2}{(4m + 1)(3m + 1)(2m + 1)^2} & \frac{2m^2(m + 1)(48m^3 + 59m^2 + 27m + 4)}{(5m + 1)(4m + 1)(3m + 1)^2(2m + 1)^2} \end{pmatrix}. $$
\end{theorem}

\begin{proof}
	By the 
	multivariate martingale central limit theorem, we have
	$$\mK_n \left(\vecM_n - \begin{pmatrix}
	1 \\ 0
	\end{pmatrix}\right) \covD \normal_2 \left({\bf 0}, \mSigma \right),$$
	where $\mSigma$ is the covariance matrix given in Lemma~\ref{Lem:multicovariance}. Use the asymptotics of $\maP_n$ and $\mQ_n$, and we arrive at
	\begin{align*}
		&\frac{1}{\sqrt{n}} \left(\begin{pmatrix}
			\frac{1}{\Gamma \left(\frac{2m + 1}{m + 1}\right)} & 0
			\\  -\frac{2(m + 1)}{m \, \Gamma\left(\frac{2m}{m + 1}\right)} & \frac{m + 1}{m \, \Gamma \left(\frac{2m}{m + 1}\right)}\end{pmatrix}
		\begin{pmatrix}
			Y_n^{(0)} \\ 2Y_n^{(1)}
		\end{pmatrix}
		+ \begin{pmatrix}
			-\frac{1}{\Gamma \left(\frac{3m + 2}{m + 1}\right)}
			\\ \frac{2(m + 1)^2}{m(3m + 1) \, \Gamma \left(\frac{2m}{m + 1}\right)}
		\end{pmatrix} \, n \right) 
		\\ &\qquad{}\, \longrightarrow \, \normal_2 ({\bf 0}, \mSigma).
	\end{align*}
	This is equivalent to the stated convergence in distribution.
\end{proof}

\section{Degree Profile}
\label{Sec:degprofile}

In this section, we study the evolution of the degree of a node as the circuit ages. Let $D^{(m)}_{j,n}$, for $j = 0, \ldots n$, be the degree of node $j$ in a preferential attachment circuit of index $m$ and of age $n$. The random variables $D^{(m)}_{j,n}$, for $j = 0, \ldots, n$, describe a profile of degrees in the random circuit. We shall cast the results in terms of Pochhammer's symbol for the rising factorial, defined as follows:
$$ \langle x \rangle_s = x(x + 1)(x + 2)\cdots(x + s - 1), $$
for any $x \in \mathbb{R}$, and any integer $s \ge 0$, with the interpretation that $\langle x \rangle_0 = 1$. 

The growth of $D^{(m)}_{j,n}$, the degree of node $j$, can be monitored via a series of $(n-j)$ different
\polya\ urns with  adjustments (that we call ``hiccups'') in between. So, the \polya\ urn is a fundamental tool in the study, and we say a quick word about it.

\subsection{\polya\ urns}
\label{Subsec:Polya}
A two-color {\em \polya\ urn scheme} is an urn 
containing balls of up to two different colors (say white and blue for example). At each point of discrete time, we draw a ball from the urn at random, 
observe its color and put it back in the urn, then execute some ball
addition rules represented 
by a $2 \times 2$ 
{\em replacement matrix}:
$$ \begin{pmatrix}
a_{1, 1} & a_{1, 2} 
\\
a_{2, 1} & a_{2, 2} 
\end{pmatrix},$$
in which the rows from top to bottom are indexed by white and blue, and  the columns from left to right are also indexed by white and blue;
entry $a_{i, k}$ is the number of balls of color $k$ that we add upon withdrawing a ball of color $i$. We refer readers to the text~\cite{Mahmoud2008} for more details.

The special case of \polya\ urns with the replacement matrix
$$ \begin{pmatrix}
1 & 0 
\\
0 & 1 
\end{pmatrix},$$
is called {\em \polya-Eggenberger urn}, and a series of urns of this flavor will appear in the study
of the degree profile. 
The \polya-Eggenberger urn has been thoroughly studied and much is known about its exact and asymptotic behavior.
The particular result we need is the following. 
Let $B_n$ be the number of blue balls in  a \polya-Eggenberger urn starting with $W_0$ white
balls and $B_0$ blue balls. With $W_0 +B_0 > 0$ (a nonempty starting condition), we have
\begin{equation*}
	\Prob (B_n = k) = \binom{m}{k} \frac{\langle W_0 \rangle_{m - k} \langle B_0 \rangle_{k}}{\langle W_0 + B_0 \rangle_m}; 
\end{equation*}
see~\cite{Mahmoud2008}.

\polya-Eggenberger urns enjoy an exchangeability property that is helpful to our analysis---the order of choosing white and blue balls in all $m$-long sequences 
does not matter in the least; all sequences
that have the same number of white balls in them have the same probability. 
For example, the sequence $WWBWBW$ has the same probability as the sequence $BBWWWW$, where $W$'s and $B$'s stand for white and blue choices.

\subsection{The degree profile via \polya\ urns}
\label{Subsec:degree}
Upon inserting the $n$th node (i.e., at time $n > j$), we allocate node $n$ and choose $m$ parent nodes (dynamically, one by one as discussed) for it. Each time node $j$ appears in the dynamic sampling, its degree goes up by one; otherwise, its degree remains the same. The transition from the $(n - 1)$th to the $n$th insertion coincides with the working of a two-color \polya-Eggenbegrer urn scheme, as we shall shortly discuss. 
We would like to alert the reader to that we are not employing one continual urn underlying the circuit growth process. Rather, for each transition to add a new node there is an underlying urn. Right after the insertion of node $n$
in the circuit, we have used $n$ different urns, and the circuit experiences ``hiccups'' in between, jolting the contents of an urn at an insertion to become the starting
conditions for the next urn used in the next insertion. 
The notion of a hiccup will be made precise in the following
paragraphs.

Let us color external nodes attached to node $j$ with white, and color those
attached to other nodes with blue. 
We reuse the notation
$W_n$ and $B_n$, but here they mean the number of white external nodes dangling out of node $j$ after $n$ insertions, and the number of blue balls dangling out of all the other nodes after $n$ insertions. As the originator does not have parents, while every other node has $m$ parents, we have the relation
\begin{equation*}
	D^{(m)}_{j,n} = 
	\begin{cases}
		W_n - 1, & j = 0;
		\\ W_n - 1 + m, & j > 0.
	\end{cases}
\end{equation*} 
For compactness, we write this double-decker expression in one line with the aid of Kronecker's delta $\delta_{.,.}$, so it takes the form
\begin{equation}
\label{degreerelation}
D_{j,n}^{(m)} = W_n - 1 + m(1 - \delta_{j,0}).
\end{equation}
Consider the sample collected for inserting node $n$ (at time $n > j$), which chooses $m$ parents for the newcomer. At this point, the circuit has $W_{n-1}$ white external nodes and $B_{n-1}$ blue external nodes. When a white external node is selected, one white external node is added to the circuit; whereas when a blue external node is selected, one blue external node is adjoined 
to the circuit. This goes on for~$m$ steps. These additions 
are just like the dynamics of a standard \polya-Eggenberger urn starting with $W_{n-1}$ white balls and $B_{n-1}$ blue balls and evolving in $m$ draws (where one ball is sampled in each drawing), see~\cite{Mahmoud2008}.

And now comes the hiccup we alluded to. The urn after $m$ ball additions (of either color) is not the right urn to model the next $m$ parent selections, without an adjustment. At the end of the collection of the sample of the $n$th insertion (i.e., after all~$m$ parents are selected), the newly added node (i.e., node $n$) acquires one additional blue external node. So, the next urn that models the transition from the $n$th insertion to the $(n + 1)$st insertion needs one extra blue ball than those at the end of adding $m$ balls to the urn used in modeling step $n$.
\begin{theorem}
	\label{Thm:degreedistribution}
	Let $D_{n, j}^{(m)}$ be the degree of node $j$ of a preferential dynamic attachment circuit with index $m$ at time $n$. For $n \ge j$, we have
	\begin{align*}
		\Prob \big( D_{n, j}^{(m)} = d \big) &= (d - m(1 - \delta_{j,0}))!\, \sum_{\substack{b_1 + \cdots +  b_{n - j} = m(n - j + 1 - \delta_{j,0}) - d \\ 0 \le b_1 \le \min\{m, (m + 1)j\}}} \left( \prod_{r = 1}^{n - j}\binom{m}{b_r} \right) \\ 
		&\qquad{} \times  \frac{\prod_{r = 1}^{n - j} \langle (m + 1)j + \sum_{\ell = 1}^{r - 1} b_\ell + r - 1 \rangle_{b_r} }{\prod_{r = 0}^{n - j - 1}\langle (m + 1)(j + r) + 1 \rangle_m}.
	\end{align*}
\end{theorem}
\begin{proof}
	We wait until node $j$ appears in the circuit, then monitor 
	the growth thereafter. Let $X_n$ be the number of blue external nodes in the circuit
	(blue balls in the urn) that appear in the sample of inserting node $n$. This variable has the distribution of the number of blue balls in the \polya-Eggenberger urn used in the $n$th node insertion. 
	We note again, the urns for the different node insertions are different. 
	For the insertion of node $j + 1$, we have an urn starting with one white external node  and $(m + 1)j$ blue external nodes. Thus, for $0 \le b_1 \le \min \{m, (m + 1)j\}$,\footnote{This condition is only needed when $j = 0$, in which cases $b_1$ 
		is $0$.} we have
	$$ \Prob(X_{j + 1} = b_1) = \binom{m}{b_1} \frac{(m - b_1)! \, \langle (m+1)j \rangle_{b_1}}{\langle \tau_j \rangle_m}.$$
	Right before inserting node $j + 2$, one additional blue external node is added to node $j + 1$ (the hiccup mentioned above). Therefore, conditioning on the event $X_{j + 1} = b_1$, before inserting node $j + 2$, we have $1 + m - b_1$ white external nodes and $(m + 1)j + b_1 + 1$ blue external nodes. According to the urn associated with this insertion, we compute
	\begin{align*}
		\Prob(X_{j + 2} = b_2 \given X_{j + 1} = b_1)  &=  \binom{m}{b_2} \frac{\langle  m - b_1 + 1 \rangle_{m - b_2} }{\langle \tau_{j + 1} \rangle_m}
		\\ &\qquad\qquad{}\times \langle (m + 1) j+ b_1 + 1 \rangle_{b_2}.
	\end{align*}
	So, we have the joint distribution of the number of external nodes that appear in the first two insertions:
	\begin{align*}
		&\Prob( X_{j + 1} = b_1, X_{j + 2} = b_2) = \Prob(X_{j+2} = b_2 \given X_{j + 1} = b_1)\, \Prob(X_{j + 1} = b_1) \\
		&\qquad = \binom{m}{b_1} \frac{(m - b_1)! \, \langle (m+1)j \rangle_{b_1}}{\langle \tau_j \rangle_m} \\
		&\qquad\qquad\qquad {} \times \binom{m}{b_2} \frac{\langle  m - b_1 + 1 \rangle_{m - b_2} \langle (m+1) j+ b_1 + 1 \rangle_{b_2}}{\langle \tau_{j + 1} \rangle_m}
		\\  &\qquad = \binom{m}{b_1} \binom{m}{b_2} \frac{(2m - b_1 - b_2)! \, \langle (m+1)j \rangle_{b_1}}{\langle \tau_j \rangle_m} \\
		&\qquad\qquad\qquad{} \times \frac{\langle (m+1) j+ b_1 + 1 \rangle_{b_2}}{\langle \tau_{j + 1} \rangle_m}.
	\end{align*}
	
	We can get the joint distribution of $X_{j + 1}, X_{j + 2}, X_{j + 3}$, by conditioning on $X_{j + 1}=b_1$ and $X_{j + 2} = b_2$,  then uncondition via the joint distribution of $X_{j + 1}, X_{j + 2}$. We can continue in this fashion to get the joint distribution of $X_{j + 1}, X_{j+2}, \ldots, X_n$. 
	We establish
	\begin{align*}
		\Prob(X_{j + 1} = b_1,  \ldots, X_n = b_{n-j}) &= \Big((n - j)m - \sum_{r = 1}^{n - j} b_r \Big)! \,\left( \prod_{r = 1}^{n - j}\binom{m}{b_r} \right) 
		\\ &\quad {}\times \frac{ \prod_{r = 1}^{n - j} \langle (m + 1)j + \sum_{\ell = 1}^{r - 1} b_\ell + r - 1 \rangle_{b_r} }{\prod_{s = j}^{n - 1}\langle \tau_s \rangle_m}.
	\end{align*}
	The exact probability distribution of the number of blue external nodes (blue balls) at time $n$ follows by summing the joint probabilities over every feasible tuple $(b_1, \ldots, b_{n-j})$ with nonnegative components adding up to $b_1 + \cdots + b_{n - j} + (m + 1)j + (n - j)$, giving us: 
	\begin{align*}
		\Prob(B_{n} = b) &= ((m + 1)n - b)!\, \sum_{b_1 + \cdots +  b_{n - j} = b - mj - n} \left( \prod_{r = 1}^{n - j}\binom{m}{b_r} \right) \\ 
		&\qquad{} \times  \frac{\prod_{r = 1}^{n - j} \langle (m + 1)j + \sum_{\ell = 1}^{r - 1} b_\ell + r - 1 \rangle_{b_r} }{\prod_{r = 0}^{n - j - 1}\langle (m + 1)(j + r) + 1 \rangle_m}. 
	\end{align*}
	The theorem follows by calculating the exact probability distribution when $b = (m + 1)n - d + m(1 - \delta_{j,0})$, which is equivalent to the probability of the event $W_n = d - m(1 - \delta_{j, 0}) + 1$, and $W_n$ translates into $D_{n,j}^{(m)}$ by~(\ref{degreerelation}).
\end{proof}

The distribution function of $D_{n,j}^{(m)}$ is unwieldy. In general, it does not give good insight into the asymptotics associated with short and long term insertions in the circuit. However, we are able to investigate the exact distribution of the degree profile of some particular nodes from simple networks. For instance, when $m = 1$ and $j = 0$, the random variable $D_{n, 0}^{(1)}$ represents the degree of the root of a PORT at time $n$. According to Theorem~\ref{Thm:degreedistribution}, we simplify the probability mass function, and arrive at
$$\Prob(D_{n, 0}^{(1)} = d) = \frac{(2n - d - 1)! \, d}{(n - d)! \, (2n - 1)!! \, 2^{n - d}}.$$
As it is not easy to use the exact distribution to do a calculation of the exact or asymptotic average or variance for larger node labels $j$ or more complicated circuits (i.e.,
larger $m$), we resort to transparent stochastic recurrence techniques. 

Consider $W_n$, the number of white external nodes at time $n$. According to the known results of \polya-Eggenberger Urn schemes (see~\cite{Mahmoud2008}), we have
\begin{align}
	\Expe[W_{n} \given \F_{n - 1,0}] &= \frac{W_{n - 1}}{\tau_{n - 1}}m + W_{n- 1},
	\label{Eq:recdegexp}
	\\ \Var[W_{n} \given \F_{n - 1,0}] &= \frac{W_{n - 1} B_{n - 1} m (m + \tau_{n - 1})}{\tau_{n  - 1}^2 (\tau_{n - 1} + 1)}. 
	\label{Eq:recdegvar}
\end{align}
We want to point out that the total number of external nodes (i.e., $\tau_n$) in the recurrences fixes the problem of one additional blue external node added to the circuit after collecting the sample of parents (after the hiccup). Solving the recurrences, we obtain the exact expectation and variance of $W_{n}$ as stated in the following theorem.
\begin{theorem}
	\label{Thm:wtilderexpvar}
	Let $D_{n,j}^{(m)}$ be the degree of node $j$ from a preferential dynamic attachment circuit with index $m$ at time $n$. For $n \ge j$, we have
	\begin{align*}
		\Expe \big[D_{n,j}^{(m)} \big] &=  \displaystyle \frac{\Gamma (n + 1) \, \Gamma \big( j + \frac{1}{m + 1} \big)}{\Gamma \big( n + \frac{1}{m + 1} \big) \, \Gamma (j + 1)} - 1 + m(1 - \delta_{j,0}),
		\\ 
		\Var\big[D_{n,j}^{(m)} \big] &= \displaystyle \frac{\Gamma(n + 1)}{((m + 1)j + 1) \, \Gamma^2 \big(n + \frac{1}{m + 1} \big) \, \Gamma^2 (j + 1) \, \Gamma \big( n + \frac{2}{m + 1} \big)}\\ 
		&{} \times \Big\{ 2((m + 1)n + 1) \, \Gamma(j + 1) \, \Gamma^2 \Big(n + \frac{1}{m + 1} \Big) \, \Gamma \Big(j + \frac{2}{m + 1}\Big)  \\ 
		&{}-((m + 1)j + 1) \Big[ \Gamma(n + 1) \, \Gamma \Big(n + \frac{2}{m + 1}\Big) \, \Gamma^2 \Big(j + \frac{1}{m + 1}\Big) \\ 
		&{} + \Gamma \Big( n + \frac{2}{m + 1}\Big) \, \Gamma \Big(n + \frac{1}{m + 1} \Big) \, \Gamma(j + 1) \, \Gamma \Big(j + \frac{1}{m + 1}\Big) \Big] \Big\}.
	\end{align*}
\end{theorem}
\begin{proof}
	We consider the external nodes in preferential attachment circuits. We obtain the recurrence for $\Expe[W_n]$ by taking another expectation of~(\ref{Eq:recdegexp}). In addition to the initial condition $W_j \equiv 1$, we get
	$$ \Expe[W_{n}] = \displaystyle \frac{\Gamma (n + 1) \, \Gamma \big( j + \frac{1}{m + 1} \big)}{\Gamma \big( n + \frac{1}{m + 1} \big) \, \Gamma (j + 1)}.$$
	According to the law of total variance (based on equations~(\ref{Eq:recdegexp}) and~(\ref{Eq:recdegvar})), and the result of $\Expe[W_n]$, we have
	\begin{align*}
		\Var[W_{n}] &= \Expe \big[\Var[W_{n - 1} \given \F_{n - 1,0}] \big] + \Var \big[ \Expe[W_{n - 1} \given \F_{n - 1,0} ] \big]
		\\ &= \frac{((m + 1)n + 1) n (m + 1)}{((m + 1)n - m + 1) ((m + 1)n - m)} \Var[W_{n - 1}] 
		\\ &\quad {} - \frac{\Gamma \big(j + \frac{1}{m + 1} \big) \, \Gamma(n) - \Gamma(j + 1) \, \Gamma \big(n - \frac{m}{m + 1} \big)((m + 1)n - m)}{((m + 1)n - m)^2 ((m + 1)n - m + 1) \, \Gamma^2 \big(n - \frac{m}{m + 1} \big) \Gamma(j + 1)^2}
		\\ &\quad\quad{} \times mn(m + 1) \, \Gamma \Big(j + \frac{1}{m + 1} \Big) \, \Gamma(n).
	\end{align*}
	The exact variance of $W_{n}$ is obtained by solving the variance recurrence above with the initial condition $\Var[W_j] = 0$.
	
	Finally, both $\Expe[W_n]$ and $\Var[W_n]$ translate to $\Expe \big[D_{n,j}^{(m)} \big]$ and $\Var \big[D_{n,j}^{(m)} \big]$ via relation~(\ref{degreerelation}).
\end{proof}

\begin{cor}
	\label{Cor:wtilderasyexpvar}
	As $n \to \infty$, the asymptotic expectation and variance of $D_{n,j}^{(m)}$ are
	\begin{align*}
		\Expe\big[D_{n,j}^{(m)} \big] &\sim \displaystyle \frac{\Gamma \big(j + \frac{1}{m + 1} \big)}{\Gamma(j + 1)} \, n^{\frac{m}{m + 1}},
		\\ \Var \big[D_{n,j}^{(m)} \big] &\sim \Big( \displaystyle \frac{2(m + 1) \, \Gamma \big(j + \frac{2}{m + 1}\big)}{((m + 1)j + 1) \, \Gamma(j + 1)} - \frac{\Gamma^2 \big(j + \frac{1}{m + 1}\big)}{\Gamma^2 (j + 1)} \Big) \, n ^{\frac{2m}{m + 1}}.
	\end{align*}
	\begin{proof}
		The asymptotic expectation and variance of $D_{n,j}^{(m)}$ are obtained by applying {\em Stirling's approximation} to the exact expectation and variance of $D_{n,j}^{(m)}$ in Theorem~\ref{Thm:wtilderexpvar}.
	\end{proof}
\end{cor}

We have an {\em exact} expression for the mean, and we can analyze it asymptotically for different regimes of $j$, whereupon we discover ``phases changes.'' For $j = j_n$, $n \to \infty$, the degree is not asymptotically much different from the case of fixed $j$, though we can simplify the gamma functions involving~$j_n$ via Stirling's approximation. For such slowly growing node labels $j_n$ (e.g., $j_n = \lfloor \sqrt{n} + 3 \rfloor$), we have
$$ \Expe \big[D_{n,j}^{(m)}\big] \sim \Big(\frac{n}{j_n}\Big)^{\frac{m}{m + 1}}.$$
If $j_n$ grows with $n$ in such a way that $j_n / n \to \theta$,
with $0 < \theta \le 1$  (e.g., $j_n = \lceil \frac 1 6 n + 3\ln n \rceil$), we have
$$ \Expe \big[D_{n,j}^{(m)}\big] \sim \frac{1}{\theta^{\frac{m}{m + 1}}} - 1 + m. $$
Note something special about the case $\theta = 1$, 
(e.g., $j_n = \lfloor n - 2 \ln \ln n\rfloor$),
where we get $\Expe \big[D_{n,j}^{(m)}\big] \sim m$. These very late arrivals join, and on average they do not recruit (outdegree 0). The only contribution to their degree is $m$ ({\em indegree}), coming from the connections to the parents.

\appendix

\section{Appendix section}
\label{Sec:app}

\subsection{Coefficients in Proposition~\ref{Prop:samplemoments}.}
\label{App:coefficients}
\begin{align*}
	\wigC_1 &= 4m(m - 1)(mn - m + n - 2),
	\\ \wigC_2 &= 4m(mn - m + n - 3)(mn - m + n - 2),
	\\ \wigC_3 &= (mn - m + n - 4)(mn - m + n - 3)(mn - m + n - 2),
	\\ \wigC_4 &= 4m(m^2n^2 - m^2n + 2mn^2 + m^2 - 7mn + n^2 + m - 6n + 10),
	\\ \wigC_5 &= 2m(mn - m + n - 2)(2mn - m + 2n - 7).
\end{align*}

\subsection{Exact moment of $\Expe[B^2_n]$ in Proposition~\ref{Prop:secmoments}.}
\label{App:secondmoment}
\begin{align*}
	\Expe[B_n^2] &= \frac{n - 1}{(mn + n - 3)(mn + n - 2)(mn + n - 1)} \left(\frac{4m^2(m + 1)^5n^4}{(3m + 1)^2(2m + 1)^2}\right.
	\\ &\qquad{}+\frac{(m + 1)^3}{(5m + 1)(4m + 1)(3m + 1)^2(2m + 1)^2}
	\\ &\qquad\qquad{}\times \bigl(4(m + 1)m(116m^4 + 47m^3 + 39m^2 + 13m + 1)n^3
	\\ &\qquad\qquad\qquad{}+ (304m^6 - 1772m^5 - 1836m^4 - 1043m^3 
	\\ &\qquad\qquad\qquad\qquad{}- 173m^2 + 7m + 1)n^2\bigr)
	\\ &\qquad{}-\frac{(m + 1)}{(5m + 1)(4m + 1)(3m + 1)^2(3m - 1)(2m + 1)^2}
	\\ &\qquad\qquad{}\times\bigl( (m + 1)(528m^8 + 6100m^7 + 156m^6 + 261m^5 
	\\ &\qquad\qquad\qquad\qquad{}+ 1731m^4 + 1594m^3 + 86m^2 - 83m - 5)n 
	\\ &\qquad\qquad\qquad{}+ 2(144m^9 - 300m^8 - 5182m^7 - 7839m^6 
	\\ &\qquad\qquad\qquad\qquad{}- 6918m^5 - 3483m^4 - 828m^3 + 99m^2 
	\\ &\qquad\qquad\qquad\qquad\qquad\qquad{}+ 58m + 3)\bigr).
\end{align*}



\begin{thebibliography}{99}
	\bibitem{Albert} 
	Barab\'{a}si., A.\ and Albert, R.\ (1999).
	Emergence of scaling in random networks. 
	{\em Science}, {\bf 286}, 509--512.
	
	\bibitem{Berger}
	Berger, N., Borgs, C., Chayes, J.\ and Saberi, A.\ (2014).
	Asymptotic behavior and distributional limits of preferential attachment graphs.
	{\em Ann.\ Probab.}, {\bf 42}, 1--40.
	
	\bibitem{Billingsley}
	Billingsley, P.\ (1995). {\em Probability and Measure}, 3rd Edition. John Wiley \& Sons, New York.
	
	\bibitem{Bollobas}	
	Bollob\'{a}s, B., Riordan, O., Spencer J.\ and  Tusn\'{a}dy, G.\ (2001). The degree sequence of a scale-free random graph process. {\em Random Structures Algorithms}, {\bf 18}, 279--290.
	
	\bibitem{Dereich}	
	Dereich, S.\ and Ortgiese, M. (2014). Robust analysis of preferential attachment models with fitness. {\em Combin.\ Probab.\ Comput.}, {\bf 23}, 386--411.
	
	\bibitem{Drmota}
	Drmota, M., Gittenberger, B.\ and Panholzer, A.\ (2008).
	The degree distribution of thickened trees.
	{\em Fifth Colloquium on Mathematics and Computer Science}, 149--161.
	
	\bibitem{Fuchs}
	Fuchs, M., Hwang, H.\ and Neininger, R.\ (2006).
	Profile of random trees: Limit theorems for random recursive trees and binary search trees.
	{\em Algorithmica}, {\bf 46}, 367--407.
	
	\bibitem{Mohan}
	Gopaladesikan, M., Mahmoud, H.\ and Ward, M.\ (2014).
	Asymptotic joint normality of counts of uncorrelated motifs in recursive trees.
	{\em Methodol.\ Comput. Appl.\ Probab.}, {\bf 16}, 863--884.
	
	\bibitem{Hwang}
	Hwang, H.\ (2007). 
	Profiles of random trees: plane-oriented recursive trees.
	{\em  Random Structures Algorithms}, {\bf 30}, 380--413.
	
	\bibitem{Kholfi}
	Kholfi, S.\ (2012).
	{\em On a class of zero-balanced urn models} (Ph.D.\ Thesis). The George Washington University, Washington, D.C.
	
	\bibitem{Kuchler}
	K\"{u}chler, U.\ and S\o rensen, M.\ (1998).
	A note on limit theorem for multivariate martingales. 
	{\em Bernoulli}, {\bf 5}, 483--493.
	
	\bibitem{Mahmoud2008}
	Mahmoud, H.\ (2009).
	{\em \polya\ Urn Models}. Chapman-Hall, Orlando, Florida.
	
	\bibitem{Mahmoud2014}
	Mahmoud, H.\ (2014). 
	The degree profile in some classes of random graphs	that generalize recursive trees. 
	{\em Methodol.\ Comput.\ Appl.\ Probab.}, {\bf 16}, 643--673.
	
	\bibitem{Mahmoud2004}
	Mahmoud, H.\ and Tsukiji, T.\ (2004).  
	Limit laws for terminal nodes in random circuits with restricted fan-out: a family of graphs generalizing binary search trees. 
	{\em Acta Inform.}, {\bf 41}, 99--110.
	
	\bibitem{Mahmoud1993}
	Mahmoud, H., Smythe, R.\ and Szyma\'{n}ski, J.\ (1993). 
	On the structure of plane-oriented recursive trees and their branches. {\em Random Structures Algorithms}, {\bf 4}, 151--176.
	
	\bibitem{Merton}
	Merton, R.\ (1968).
	The Matthew effect in science. 
	{\em Science}, {\bf 159}, 56--63.
	
	\bibitem{Moler}
	Moler, J., Plo, F.\ and Urmeneta, H.\ (2013).
	A generalized \polya\ urn and limit laws for the number of outputs in a family of random circuits. {\em TEST}, {\bf 22}, 46--61.
	
	\bibitem{Ostroumova}
	Ostroumova, L., Ryabchenko, A.\ and Samosvat, E.\ (2013). Generalized preferential attachment: tunable power-law degree distribution and clustering coefficient. {\em Algorithms and models for the web graph}, 185--202. 
	
	\bibitem{Pekoz}
	Pek\"{o}z, E., R\"{o}llin, A.\ and Ross, N. (2013). 
	Degree asymptotics with rates for preferential attachment random graphs. {\em Ann.\ Appl.\ Probab.}, {\bf 23}, 1188--1218.
	
	\bibitem{Rollin}             
	Pek\"{o}z, E., R\"{o}llin, A.\ and Ross, N. (2017). 
	Joint degree distributions of preferential attachment random graphs. {\em Adv.\ in Appl.\ Probab.}, {\bf 49}, 368--387.
	
	\bibitem{Resnick}   
	Resnick, S.\ and Samorodnitsky, G.\ (2016). 
	Asymptotic normality of degree counts in a preferential attachment model. {\em Adv.\ in Appl.\ Probab.}, {\bf 48}, 283--299.
	
	\bibitem{Ross}
	Ross, N.\ (2013).
	Power laws in preferential attachment graphs and Stein's method for the negative binomial distribution. {\em Adv.\ in Appl.\ Probab.}, {\bf 45}, 876--893.
	
	\bibitem{Tsukiji}
	Tsukiji, T.\ and Mahmoud, H.\ (2001). 
	A limit law for outputs in random circuits. 
	{\em Algorithmica}, {\bf 31}, 403--412.
	
	\bibitem{Wang}
	Wang, T.\ and Resnick, S. (2017). Asymptotic normality of in- and out-degree counts in a preferential attachment model. {\em Stoch.\ Models}, {\bf 33}, 229--255. 
	
	\bibitem{ZhangANALCO}	
	Zhang, P.\ (2016). 
	On terminal nodes and the degree profile of preferential dynamic attachment circuits. 
	{\em Proceedings of the Thirteenth Workshop on Analytic Algorithmics and Combinatorics (ANALCO)}, 80--92.
\end{thebibliography}


\end{document}